\documentclass[11pt]{amsart}
\usepackage[normalem]{ulem}

\PassOptionsToPackage{table}{xcolor}
\usepackage{stmaryrd}
\expandafter\def\csname opt@stmaryrd.sty\endcsname
{only,shortleftarrow,shortrightarrow}
\usepackage{extpfeil}
\usepackage{amsmath,amssymb,amsthm}
\usepackage{aliascnt}
\usepackage{varioref}
\usepackage{hyperref}
\usepackage{thmtools}
\usepackage[capitalize,nameinlink,noabbrev]{cleveref}
\hypersetup{colorlinks=true,urlcolor=blue,citecolor=blue,linkcolor=blue}
\usepackage{courier}
\usepackage{tikz}
\usepackage{tikz-cd}
\usepackage{quiver}
\usetikzlibrary{calc,matrix,arrows,decorations.markings}
\usepackage{array}
\usepackage{color}
\usepackage{enumerate}
\usepackage{nicefrac}
\usepackage{listings}
\usepackage{siunitx}
\usepackage{seqsplit}
\usepackage{longtable}
\usepackage{colonequals}
\usepackage{array, multirow}
\newcommand{\Q}{\mathbb{Q}}

\newcommand{\R}{\mathbb{R}}
\newcommand{\C}{\mathbb{C}}

\newcommand{\Z}{\mathbb{Z}}

\DeclareFontFamily{U}{wncy}{}
    \DeclareFontShape{U}{wncy}{m}{n}{<->wncyr10}{}
    \DeclareSymbolFont{mcy}{U}{wncy}{m}{n}
    \DeclareMathSymbol{\Sh}{\mathord}{mcy}{"58} 

\newtheorem{theorem}{Theorem}[section]

\newtheorem{lemma}[theorem]{Lemma}
\newtheorem{proposition}[theorem]{Proposition}
\newtheorem{corollary}[theorem]{Corollary}

\theoremstyle{definition}
\newtheorem{definition}[theorem]{Definition}
\newtheorem{example}[theorem]{Example}

\newtheorem{Conjecture}[theorem]{Conjecture}
\newtheorem{remark}[theorem]{Remark}

\usepackage{xcolor}

\title{Replacement dynamics of binary quadratic forms}

\begin{document}

\author{\sc Raghav Bhutani}
\address{Raghav Bhutani \\
University of Illinois Chicago\\
USA}
\urladdr{}
\email{raghavbhutani17@gmail.com}

\author{\sc Frederick Saia}
\address{Frederick Saia \\
University of Illinois Chicago\\
USA}
\urladdr{https://fsaia.github.io/site/}
\email{freddy.v.saia@gmail.com}

\begin{abstract}
For an $S$-valued function $f$ of $m \geq 1$ variables we consider the dynamical process in which the output $f(\overline{v})$ replaces exactly one entry of the input $\overline{v} \in S^m$ at each step. This can be viewed as a special case of multivariate polynomial semigroup dynamics, and our study focuses on periodic vectors with respect to this process. We define a stratification of periodic vectors according to their type, and characterize types for which the determination of periodic vectors comes down to dynamics of univariate polynomials. We then restrict to the case of a diagonal binary quadratic form $f$ over $\mathbb{Q}$, and classify rational periodic vectors for all types of period up to $5$. This includes two types which do not arise from the univariate case. 
\end{abstract}

\maketitle

\section{Introduction}

\subsection{Motivation: the univariate case}

Let $f \in \Q[x]$ be a univariate polynomial function with rational coefficients. The \textbf{rational periodic points} of $f$ of period $N \in \Z$ are those $x \in \Q$ satisfying
\[ x = f^{(N)}(x) := (f \circ \cdots \circ f)(x) \]
and not satisfying $x = f^{(m)}(x)$ for any $1 \leq m < N$. These points are central objects of study in the field of arithmetic dynamics, which studies the behavior of such polynomials under iteration. 

A particular example of interest is that in which $f$ is a degree $2$ polynomial. Here, consideration can be reduced, via a linear change of variables, to the periodic points of $f(x) = x^2+c$ for $c \in \Q$. Despite $2$ being the first non-trivial degree to study, one already encounters difficult and interesting questions; this case has seen considerable focus in the last few decades, beginning with work of Morton \cite{Mor1,Mor2}. In the first of these works, Morton handles the period $N=3$ case (with the $N=1,2$ cases being easier and previously handled). In the latter, Morton shows the non-existence of rational period $4$ points for quadratic polynomials \cite[Thm. 4]{Mor2}. (Some of Morton's work has overlap with an article of Walde--Russo \cite{WR}, to which we also refer the reader.) 

In \cite[Thm. 1]{FPS}, Flynn--Poonen--Schaefer prove that quadratic polynomials also have no rational periodic points of order $5$. Poonen \cite[Conj. 2]{Poonen} has conjectured that this is the case for all $N \geq 4$: 

\begin{Conjecture}[Poonen]\label{Poonen_conj}
If $N \geq 4$, then there is no quadratic polynomial $f(x) \in \Q[x]$ with a rational point of period $N$. 
\end{Conjecture}

Hutz--Ingram provided computational evidence towards \cref{Poonen_conj}, proving that the claim holds for all $f(x) = x^2+c$ with $c \in \Q$ of height up to $10^8$ \cite[Prop. 1]{HI}. Stoll \cite[Thm. 7]{Stoll} determined that the $N=6$ case of \cref{Poonen_conj} follows from the Birch and Swinnerton--Dyer conjecture for a specific abelian fourfold, namely the Jacobian of an algebraic curve over $\Q$ which parameterizes orbits of periodic points of quadratic polynomials of period $6$. This curve is referred to as a dynamical modular curve, in analogy to the modular curves which parameterize torsion of elliptic curves. The works of Morton, Flynn--Poonen--Schaefer, and others who have worked on this problem over $\Q$ and over higher degree number fields also take the vantage of studying the arithmetic of dynamical modular curves. 

If $f$ is a polynomial in $m > 1$ variables, then there
is no natural notion of ``dynamics'' for $f$; we cannot literally iterate $f$ like we do in the $m=1$ case. We consider in this work a \emph{specific} generalization of the single-variable case that leads to interesting  Diophantine questions. We study rational periodic vectors under this generalization, which we dub ``replacement dynamics'' for the purposes of this work, in the case of diagonal binary quadratic forms.

\subsection{Replacement dynamics}\label{dynamic_reps_section}

Let $S$ be a set, let $m$ be a positive integer, and let 
\[ f: S^m \longrightarrow S \]
denote an $S$-valued function in $m$ variables. Our ``replacement dynamics'' framework is as follows: for $\overline{v} = (v_1,\ldots, v_m) \in S^m$, and for each $j \in \{1, \ldots, m\}$ let 
\begin{equation}\label{iterate_fcn_notation}
f_{(j)}(\overline{v}) =  f_{(j)}(v_1,\ldots, v_m) := (v_1,\ldots,v_{j-1},f(\overline{v}),v_{j+1},\ldots,v_{m})
\end{equation}
denote the vector obtained by replacing the $j^\textnormal{th}$ entry of $\overline{v}$ with $f(\overline{v})$. Let $V_0 := \{\overline{v}\}$, and recursively define for each positive integer $k$
\[ V_k(f,\overline{v}) = \left\{f_{(j)}(\overline{w}) \mid \overline{w} \in V_{k-1} \textnormal{ and } j \in \{1,\ldots,m\}\right\}. \]
We define the directed graph $\mathcal{G}(f,\overline{v})$ to be that with vertex set in correspondence with the set
\[ V(f,\overline{v}) = \bigcup_{k \geq 0} V_k(f,\overline{v}), \]
which we refer to as the set of all iterates of $\overline{v}$ under $f$. For $\overline{v_1}, \overline{v_2} \in V(f,\overline{v})$, we have a directed edge in $\mathcal{G}(f,v)$ from the vertex corresponding to $\overline{v_1}$ to that corresponding to $\overline{v_2}$ if there is at least one index $j \in \{1, \ldots, m\}$ such that 
\[ f_{(j)}(\overline{v_1}) = \overline{v_2}. \]

\begin{definition}\label{periodic_def}
    Let $f : S^m \rightarrow S$ be a function and let $\overline{v} \in S^m$. 
    \begin{itemize}
        \item We call $\overline{v}$ a \textbf{periodic vector for $f$} if the vertex in $\mathcal{G}(f,\overline{v})$ corresponding to $\overline{v}$ is contained in a cycle. 
        \item We call $\overline{v}$ \textbf{pre-periodic} for $f$ if there exists a periodic vector $\overline{w}$ for $f$ for which there is a path in $\mathcal{G}(f,V_0)$ from the vertex corresponding to $\overline{v}$ to that corresponding to $\overline{w}$. 
    \end{itemize}
\end{definition}

\begin{remark}
    If $\overline{w} \in V(f,\overline{v})$, then $\mathcal{G}(f,\overline{w})$ is a subgraph of $\mathcal{G}(f,\overline{v})$. Moreover, with the same assumption, the statement that $\overline{w}$ is periodic for $f$ is equivalent to the existence of a cycle in $\mathcal{G}(f,\overline{v})$ containing the vertex corresponding to $\overline{w}$. 
\end{remark}

\begin{remark}
    We can view our framework in the context of the dynamics of (multivariate) polynomial semigroups, as in, e.g., \cite{HM96, Mello20, HS23,BHZ25}, as follows: given a set $M$ of polynomial maps $S^m \to S^m$, one considers the dynamical procedure in which at each step the input vector $\overline{v} \in S^m$ is replaced with $g(\overline{v})$ for some $g \in M$. In other words, this is the study of the dynamics of the semigroup
    \[ G_M := \{g_1 \circ g_2 \circ \cdots \circ g_n \mid n \geq 1 \text{ and } g_i \in M \text{ for all } i \in \{1,\ldots, n\}\} \]
    generated by $M$ under composition. For a fixed $f : S^m \to S$, our ``replacement dynamics'' is precisely this procedure for the polynomial semigroup $G_M$ generated by the set
    \[ M = \{f_{(j)} \mid j \in \{1,\ldots, m\}\}, \]
    with $f_{(j)} : S^m \to S^m$ as defined in \cref{iterate_fcn_notation}. The arithmetic study of the dynamics of (multivariate) polynomial semigroups, even of quadratic polynomials in particular, are of great interest, though it has seen the most traction in special cases. The current work can be seen as a restriction to a case where some headway can be made. 
\end{remark}

Our process admits a stratification of periodic vectors of a fixed period according to what we will call their type.

\begin{definition}\label{types_def}
    Let $N$ and $m$ be positive integers. 
    \begin{itemize}
        \item We define a \textbf{period type of period $N$ and of length $m$} to be an ordered $N$-tuple 
    \[ \mathbf{t} = (t_1,t_2,\ldots,t_N) \in \{1,2,\ldots,m\}^N \] 
    in the letters $1$ through $m$. We call any type $\mathbf{t'}$ of the form $\mathbf{t'} = (t_1,\ldots, t_{N'})$ with $1 \leq N' \leq N$ a \textbf{subtype} of $\mathbf{t}$. 
    
        \item We say that a periodic vector $\overline{v}$ for $f$ \textbf{is of type} $\mathbf{t}$ if there is a directed cycle of length $N$ in $\mathcal{G}(f,\overline{v})$ which begins at the vertex corresponding to $\overline{v}$ and has its $i^\textnormal{th}$ edge induced by replacement of the ${t_i}^\textnormal{th}$ entry in the input to $f$ for each $1 \leq i \leq N$. 
    \end{itemize}
\end{definition}

We further spell out this definition, both to ensure clarity and to introduce notation. For a type $\mathbf{t} = (t_1,\ldots,t_N)$ of period $N$ vectors, let $f_\mathbf{t} : S^m \rightarrow S^m$ denote the function obtained as follows: in the $N = 1$ case, 
\[ f_\mathbf{t} = f_{(t_1)} \]
is as defined in \cref{iterate_fcn_notation}. For $2 \leq i \leq m$, we define $f_{(t_1,\ldots, t_i)}$ inductively via
\begin{equation}\label{inductive_def_of_iterates} f_{(t_1,\ldots,t_i)} := f_{(t_i)} \circ f_{(t_1,\ldots,t_{i-1})}. 
\end{equation}
With this notation, a vector $\overline{v}$ is a period $N$ vector of type $\mathbf{t}$ for $f$ precisely when we have the equality
\begin{equation}\label{periodic_type_t_eqn}
 f_{\mathbf{t}}(\overline{v}) = \overline{v} 
\end{equation}
and there is no subtype $\mathbf{t'}$ of $\mathbf{t}$ of period $N'<N$ such that $f_{\mathbf{t'}}(\overline{v}) = \overline{v}$. 


We prove in \cref{type_classes_section}, specifically in \cref{sigma_bijection_lemma} and \cref{tau_bijection_lemma}, that it is sufficient to consider classes of types modulo an equivalence relation defined in \cref{type_class_def}. Following this setup, we focus our study on a specific family of functions of two-variables, namely diagonal binary quadratic forms over $\Q$, which highlights new phenomena in our setting while also connecting back to the work accomplished in the single-variable setting for quadratic polynomials.

\subsection{Binary quadratic forms and main results}\label{main_results_section}

In \cref{modular_polynomials_section}, we define for a fixed binary quadratic form $f = Cx^2 + Dy^2$ with $C,D \ne 0$ and a fixed binary (length $2$) type $\mathbf{t}$ a corresponding moduli space $Y_{f,\mathbf{t}}$ over $\Q$ whose points over $\overline{\Q}$ generically are in correspondence with periodic vectors for $f$ of type $\mathbf{t}$. Our aim is then to understand the rational points on $Y_{f,\mathbf{t}}$ for specified types $\mathbf{t}$.

We make the following definition, which identifies types $\mathbf{t}$ for which the process of determining all periodic vectors of type $\mathbf{t}$ will essentially come down to considering periodic points of (families of) single-variable quadratic polynomials. 

\begin{definition}\label{univariate_def}
    Let $\mathbf{t} = (t_1, \ldots, t_N) \in \{L,R\}^N$ be a binary type of period $N$ vectors. We say that $\mathbf{t}$ is a \textbf{univariate type} if there exists an integer $N'$ with $1 \leq N' \leq N$ such that $\mathbf{t}$ is equivalent (per \cref{type_class_def}) to the type
    \[ (\underbrace{L,\ldots, L}_{N'},\underbrace{R, \ldots, R}_{N-N'}) \]
    corresponding to $N'$ replacements on the left followed by $N-N'$ replacements on the right. 
\end{definition}

In \cref{univariate_prop}, we prove that if $\mathbf{t}$ is univariate as in the above definition with $N' < N$, then there are no periodic vectors of type $\mathbf{t}$. We additionally prove, from the results of \cite{Mor2, FPS} and \cite[Conj. 2]{Poonen}, that there are no periodic points of type $\mathbf{t}$ for $f$ a binary form as above and $\mathbf{t}$ a univariate type of period $4$ or $5$ and that there are conjecturally none of univariate type $\mathbf{t}$ of any period greater than or equal to $4$. 

All binary period types of period $N<4$ are univariate. In \cref{L_section}, \cref{LL_section}, and \cref{LLL_section}, we classify rational periodic vectors of univariate type for $f$ as above of periods $1, 2,$ and $3$, respectively (see \cref{self-loops-prop}, \cref{LL_prop}, and \cref{LLL_prop}). Such vectors can be thought of as in correspondence with periodic points of certain families of the single-variable quadratic polynomials studied in \cite{WR,Mor1}. 

Our main results are classifications of periodic vectors for diagonal binary quadratic forms $f$ in the two non-univariate cases of period at most $5$. 

\begin{theorem}\label{main_theorem}
Let $f(x,y) = Cx^2+Dy^2$ be a fixed binary quadratic form over $\Q$ with $C, D \neq 0$. 
\begin{enumerate}
    \item The form $f$ has a periodic vector over $\overline{\Q}$ of type $(L,R,L,R)$ if and only if $C \ne D$, and there is a one-to-one correspondence between such vectors and the roots of the degree $4$ polynomial $R_{C,D}(y) \in \Z[y]$ defined in \cref{LRLR_poly}:
    \begin{align*}
        \{\textnormal{roots of } R_{C,D}(y) \} &\longrightarrow \{\textnormal{periodic vectors of type } (L,R,L,R) \textnormal{ for } f \} \\
        y &\longmapsto (G_{C,D}(y),y)\; ,
    \end{align*}
    where $G_{C,D}(y) \in \Q[y]$ is an explicit polynomial. 

    In particular, there are at most $4$ periodic vectors of this type for $f$ over $\C$, the coordinates of a periodic vector of this type for $f$ generate a number field of degree at most $4$ (with degree $4$ being the generic behavior), and $f$ has a rational periodic vector over a number field $K$ if and only if $R_{C,D}(y)$ has a root over $K$. \\

    \item  Each periodic vector of type $(L,L,R,L,R)$ for $f$ over $\overline{\Q}$ corresponds to a root of the polynomial $S_{C,D}(y) \in \Z[y]$ defined in \cref{LLRLR_poly}, which has degree $10$ unless $C = -D$ or $c^2-3cd+4d^2 = 0$. If $C \ne D$, then there is a one-to-one correspondence
    \begin{align*}
        \{\textnormal{roots of } S_{C,D}(y) \} &\longrightarrow \{\textnormal{periodic vectors of type } (L,L,R,L,R) \textnormal{ for } f\} \\
        y &\longmapsto (H_{C,D}(y),y)\; ,
    \end{align*}
    where $H_{C,D}(y) \in \Q[y]$ is an explicit polynomial. 

    In particular, there are at most $10$ periodic vectors of this type for $f$ over $\C$, the coordinates of a periodic vector of type $(L,L,R,L,R)$ for $f$ generate a number field of degree at most $10$ (with degree $10$ being the generic behavior), and if $f$ has a periodic vector of this type over a number field $K$ then $S_{C,D}(y)$ must have a root over $K$. 
\end{enumerate}
\end{theorem}

The proofs of these results appear in \cref{LRLR_section} and \cref{LLRLR_section}; see \cref{LRLR_thm} and \cref{LLRLR_thm}. From part (1) of \cref{main_theorem}, we prove in analogy to \cref{Poonen_conj} the following result (proven and restated in \cref{LRLR_section} as \cref{LRLR_thm_rationals}).
\begin{theorem}
    There exist no rational numbers $C,D \neq 0$ so that the binary quadratic form $f(x,y) = Cx^2+Dy^2$ has a periodic vector of type $(L,R,L,R)$ over $\Q$. 
\end{theorem}

We remain unsure of whether there exist such binary quadratic forms with rational periodic vectors of type $(L,L,R,L,R)$, but using part (2) of \cref{main_theorem} we are able in \cref{LLRLR_section} to equate the existence of such a vector to the existence of a smooth rational point on a specified projective curve of geometric genus $X$. Computational searches for rational points on $X$ lead us to conjecture that this period type also does not admit rational periodic vectors (see \cref{LLRLR_conjecture}).

All computations mentioned in this work were performed using the Magma computer algebra system \cite{magma}, and all relevant code can be found at the repository \cite{Repo}. 

\section*{Acknowledgements}

This project was completed as part of the Liberal Arts and Sciences Undergraduate Research Initiative at the University of Illinois Chicago. We warmly thank the LASURI program for their support. We thank Nathan Jones for helpful comments on a earlier version of this work, and we graciously thank the anonymous referee for a key observation leading to the proof of \cref{LRLR_thm_rationals}.   

\section{Period type classes}\label{type_classes_section}

For this section, we fix positive integers $m$ and $N$, fix a set $S$, and let $f: S^m \to S$ be an $S$-valued function in $m$ variables. 

We first note that our iterating functions for types respect concatenation. For $\mathbf{t} = (t_1,\ldots,t_N)$ and $\mathbf{s} = (s_1,\ldots,s_M)$ types of vectors in $m$ variables with periods $N$ and $M$, respectively, we let   \[ \mathbf{t}\cdot \mathbf{s} = (t_1,\ldots,t_N,s_1,\ldots,s_M) \]
denote the concatenation of $\mathbf{t}$ and $\mathbf{s}$. 
\begin{lemma}
    With $\mathbf{t}$ and $\mathbf{s}$ as above, we have
    \[ f_{\mathbf{t} \cdot \mathbf{s}} = f_{\mathbf{s}} \circ f_{\mathbf{t}}. \]
\end{lemma}
\begin{proof}
    By the inductive definition of the polynomials corresponding to a type, it suffices to consider the case in which $\mathbf{t} = (t)$ and $\mathbf{s} = (s)$ are both types of period $1$ points. The equality $f_{(t,s)} = f_{(s)}(f_{(t)})$ is then immediate from the construction of these polynomials in \cref{inductive_def_of_iterates}.
\end{proof}

We define a rotation operator $\sigma$ which acts on a type $\mathbf{t} = (t_1,t_2, \ldots, t_N)$ for period $N$ vectors as follows:
\[ \sigma \mathbf{t} = \sigma (t_1, t_2,\ldots, t_{N-1}, t_N) = (t_N,t_1,t_2,\ldots,t_{N-1}). \]
This provides a group action by 
\[ \langle \sigma \rangle = \{\textnormal{Id},\sigma, \ldots, \sigma^{N-1}\} \cong \mathbb{Z}/N\mathbb{Z} \]
on the set of all types of period $N$ vectors, where $\sigma^j$ denotes repeated action $j$ times by $\sigma$ and $\textnormal{Id}$ denotes the identity function on the set of types of period $N$ vectors. 

\begin{proposition}\label{sigma_bijection_lemma}
    For any type $\mathbf{t}$ of period $N$ vectors and any $j \in \mathbb{Z}$, there is a bijection between the set of type $\mathbf{t}$ vectors for $f$ and the set of type $\sigma^j\mathbf{t}$ vectors for $f$. 
\end{proposition}
\begin{proof}
It suffices to prove the claim for $j = 1$, which follows from the following mutually inverse bijections:
\begin{align*} 
\{\textnormal{type $\mathbf{t}$ vectors for $f$}\} &\longleftrightarrow \{\textnormal{type $\sigma \mathbf{t}$ vectors for $f$}\} \\
\overline{v} &\longmapsto f_{(t_1)}(\overline{v}) \\
f_{(t_2,\ldots, t_{N})}(\overline{w}) &\longmapsfrom \overline{w}. 
\end{align*}
\end{proof}

\begin{remark}
    Note that it is important that we specified the type $\mathbf{t}$, rather than just the period $N$, in \cref{sigma_bijection_lemma}; the period may not identify a unique cycle to which a periodic vector belongs, while the type explicitly does. 
\end{remark}

Additionally, we have an action of the symmetric group on $m$ letters $S_m$ as follows: for $\mathbf{t}$ a type of period $N$ vectors and $\tau \in S_m$,
\[ \tau \mathbf{t} = \tau (t_1,\ldots, t_N) = (\tau(t_1), \ldots, \tau(t_N)). \]
We will also, by abuse of notation, let $\tau$ denote the function $S^m \rightarrow S^m$ obtained by permutation of a vector's entries in the following manner:
\[ \tau(v_1,\ldots, v_m) = (v_{\tau(1)}, \ldots, v_{\tau(m)}). \]

\begin{proposition}\label{tau_bijection_lemma}
    For $\tau \in S_m$ and $\mathbf{t}$ as above, there is a bijection between the set of type $\mathbf{t}$ vectors for $f$ and the set of type $\tau \mathbf{t}$ vectors for $f \circ \tau$.
\end{proposition}
\begin{proof}
    We have the following bijections:
\begin{align*} 
\{\textnormal{type $\mathbf{t}$ vectors for $f$}\} &\longleftrightarrow \{\textnormal{type $\tau \mathbf{t}$ vectors for $f \circ \tau$}\} \\
\overline{v} &\longmapsto \tau^{-1} (\overline{v}) \\
\tau (\overline{w}) &\longmapsfrom \overline{w}. 
\end{align*}
The maps are clearly mutually inverse, so we just need to show they actually have the claimed codomains. Indeed, $(f \circ \tau) (\tau^{-1}(\overline{v})) = f(\overline{v})$, and hence for $t_1 \in \{1,\ldots, m\}$ and $\overline{v} = (v_1,\ldots, v_m)$ we have that
\begin{align*} 
(f \circ \tau)_{\tau (t_1)}(\tau^{-1} (\overline{v})) &= (f \circ \tau)_{\tau (t_1)}(v_{\tau^{-1}(1)},\ldots,\underbrace{v_{t_1}}_{\tau(t_1)^\textnormal{th} \textnormal{ entry}},\ldots,v_{\tau^{-1}(m)}) \\
&= (v_{\tau^{-1}(1)},\ldots,f(\overline{v}),\ldots, v_{\tau^{-1}(m)}) \\
&= \tau^{-1} (f_{(t_1)}(\overline{v})).
\end{align*}
By the inductive definition of $f_\mathbf{t}$ for $\mathbf{t} \in \{1,\ldots,m\}^N$, we then have 
\[ (f \circ \tau)_{\tau \mathbf{t}}(\tau^{-1} (\overline{v})) = \tau^{-1} (f_\mathbf{t}(\overline{v})). \]
Thus, if $\overline{v}$ is a type $\mathbf{t}$ vector we have 
\[ (f \circ \tau)_{\tau \mathbf{t}}(\tau^{-1} (\overline{v})) = \tau^{-1} (\overline{v}). \]
\end{proof}

The actions of $\langle \sigma \rangle$ and of $S_m$ on the set of type $\mathbf{t}$ vectors for $f$ commute, giving an action of $\langle \sigma \rangle \times S_m$ on this set. Note by the above that this group action preserves the period corresponding to a type. We define type classes to be classes of types modulo this action:
\begin{definition}\label{type_class_def}
A \textbf{period type class of period $N$ in $m$ variables} is an element of the quotient set
\[ \{1,\ldots, m\}^N/\left(\langle \sigma \rangle \times S_m \right). \]
We say two period types are \textbf{equivalent} if they have the same image under the natural quotient map to the set of type classes.  
\end{definition}
In common combinatorial terminology, our type classes are necklaces in $m$ colors modulo permutations on the color set (which seem to have first been studied in \cite{Fine, GR61}). The length of a representative necklace is the period corresponding to the type class.

\subsection{The binary case}

In what follows, we will hone in on the binary case of $m=2$. In this setting, we will write a type $\mathbf{t} \in \{1,2\}^N$ of period $N$ vectors instead as an element of $\{L,R\}^N$ via the bijection sending $1$ to $L$ and $2$ to $R$. That is, $L$ will stand for replacement on the left, and $R$ for replacement on the right. 

In this case, the number of type classes of period $N$ is given by \cite[\S 6]{Fine}: 
\[ \#\left( \{1,\ldots, m\}^N/\left(\langle \sigma \rangle \times S_m \right) \right) = \sum_{d \mid N} \frac{\phi(2d)\cdot 2^{N/d}}{2N}, \]
where $\phi$ denotes Euler's totient function. The sequence of these counts as a function of the period $N$ is \cite[Sequence A000013]{OEIS}.

\section{Replacement Dynamics for Binary quadratic forms}\label{binary_forms_section}

As mentioned in the introduction, there has been significant work on rational periodic points of univariate polynomials of degree $2$. Even in this degree, the question gets difficult rather quickly; while \cref{Poonen_conj} states that there are no rational period $N$ points for quadratic polynomials when $N \geq 4$, this has only been proven unconditionally for $N=4,5$ and conditionally for $N=6$. 

We also stick to consideration of the degree $2$ case here, specifically in the case of two variables. Moreover, we study only a subfamily of the full family of binary degree $2$ polynomials, namely diagonal binary quadratic forms over $\mathbb{Q}$. That is, we consider the family of polynomials
\[ f(x,y) = Cx^2+Dy^2 \]
with $C,D \in \mathbb{Q}$ both nonzero, and we ask about their rational periodic vectors under replacement dynamics. The restriction to this family, as opposed to consideration of all quadratic polynomials in two-variables over $\Q$, cuts down the dimension of the varieties we consider in this work. We emphasize that this is the only reason for the restriction in this work; we do not make notable use of the theory of quadratic forms. 

\subsection{Dynamic modular polynomials}\label{modular_polynomials_section}

Let $\mathbf{t} = (t_1,\ldots, t_N) \in \{L,R\}^N$ be a binary period type of period $N$, and consider the family of diagonal binary quadratic forms
\[ f(x,y) = cx^2 + dy^2 \]
over $\Q$. The equality of vectors
\begin{equation}\label{dyn_poly_eqn} 
f_{\mathbf{t}}(\overline{v}) - f(\overline{v}) = \overline{0}
\end{equation}
in the indeterminates $x,y,c,d$ defines two affine varieties in the affine space $\mathbb{A}^4_{\mathbb{Q}}$, each of which is a family of varieties of dimension at least one over $\mathbb{A}^2_{\mathbb{Q}}$ via projection to the $c$ and $d$ coordinates. A rational point $(x,y,c,d) = (v_1,v_2,C,D) \in \mathbb{Q}^4$ satisfying both equations corresponds to a vector $(v_1,v_2)$ which is periodic for $f = Cx^2-Dy^2$, and is either periodic of type $\mathbf{t}$ or has period dividing $N$. 

Taking another vantage, we can consider $f$ as a polynomial in two variables $x$ and $y$ over the function field $K = \mathbb{Q}(c,d)$. From this perspective, \cref{dyn_poly_eqn} defines two affine varieties of dimension at least one over $K$, at least one of which is a curve. More specifically, let $\pi_L, \pi_R : \mathbb{Q}^2 \to \mathbb{Q}$ be the projections onto the first and second entry, respectively. \cref{dyn_poly_eqn} is equivalent to the system of equalities
\begin{align*} 
&\pi_L(f_\mathbf{t}(x,y)) = x , \qquad \textnormal{ and } \\
&\pi_R(f_\mathbf{t}(x,y)) = y,
\end{align*}
and we have factorizations of the form
\begin{align*}
    &\pi_L(f_\mathbf{t}(x,y)) - x = \Phi_{\mathbf{t},L}(x,y) \cdot \Psi_{\mathbf{t},L}(x,y) , \\
    &\pi_R(f_\mathbf{t}(x,y)) - y = \Phi_{\mathbf{t},R}(x,y) \cdot  \Psi_{\mathbf{t},R}(x,y),
\end{align*}
where $\Phi_{\mathbf{t},L}(x,y)$ and $\Phi_{\mathbf{t},R}(x,y)$ are the monic polynomials over $K$ whose common roots $(x,y)$ generically correspond to the periodic vectors of type $\mathbf{t}$ for $f$. These are defined inductively as follows: for $t_1 \in \{L,R\}$ we have
\begin{align*} 
\Phi_{(t_1),L} &= \pi_L(f_{(t_1)}(x,y)) - x \\
\Phi_{(t_1),R} &= \pi_R(f_{(t_1)}(x,y)) - y. 
\end{align*}
For general $\textbf{t}$ of period $N$, we then define $\Phi_{\mathbf{t},L}$ to be the polynomial obtained by dividing $\pi_L(f_{\mathbf{t}}(x,y)) - x$ by any factors of the form $\Phi_{\mathbf{t'},L}$, for $\mathbf{t'}$ a subtype of $\mathbf{t}$ of period $N'<N$, that appear in its factorization over $K$, and similarly for $\Phi_{\mathbf{t},R}$ with $L$ and $x$ replaced with $R$ and $y$, respectively. The description above then follows from our construction.
\begin{lemma}
Let $\mathbf{t}$ be a type of period $N$ vectors. All but finitely many of the common solutions to the specializations of $\Phi_{\mathbf{t},L}$ and $\Phi_{t,R}$ to values $C, D$ of $c$ and $d$ correspond to periodic vectors of type $\mathbf{t}$ for $f(x,y) = Cx^2+Dy^2$. 
\end{lemma}
\begin{proof}
    A common solution to these polynomials is a vector $\overline{v} = (x,y)$ satisfying \cref{dyn_poly_eqn}. This vector is then necessarily periodic for $f$ of type $\mathbf{t'}$ for some subtype $\mathbf{t'}$ of $\mathbf{t}$. If $\overline{v}$ were in fact periodic of type $\mathbf{t'}$ for some $\mathbf{t'}$ of period $N'<N$, then it would also satisfy the polynomials $\Phi_{\mathbf{t'},L}$ and $\Phi_{\mathbf{t'},R}$. By construction, the four equations coming from the types $\mathbf{t}$ and $\mathbf{t'}$ have finitely many common solutions. Since there are finitely many subtypes of $\mathbf{t}$, the claim follows. 
\end{proof}

We refer to the polynomials $\Phi_{\mathbf{t},i}$ for $i \in \{L,R\}$ as \textbf{modular polynomials}, in analogy to the modular polynomials defining modular curves, which parameterize elliptic curves with torsion structure, and their dynamical analogues in the single-variable quadratic setting as studied in, e.g., \cite{Mor1,WR,Mor2,FPS,Stoll}. In the case of univariate types, the polynomials $\Psi_{\mathbf{t},i}$ come from products of poylnomials $\Phi_{\mathbf{s},i}$ for certain types $\mathbf{s}$ of period strictly less than $N$, as we will soon spell out. 

Let 
\[ Y_{\mathbf{t},i} : \Phi_{\mathbf{t},i} = 0 \]
be the affine variety in $A^2_K$ defined by $\Phi_{\mathbf{t},i}$ for $i \in \{L,R\}$, and let 
\[ Y_{\mathbf{t}} = Y_{\mathbf{t},L} \cap Y_{\mathbf{t},R}. \]

We may specialize by fixing values $c = C$ and $d = D$ in $\Q$, i.e., by fixing our binary quadratic form $f = Cx^2+Dy^2$, to get specializations $\Phi_{f,\mathbf{t},i}$ of $\Phi_{\mathbf{t},i}$ for $i=L,R$. Each of these polynomials either defines a curve in the affine plane $A^2_\Q$ in variables $x$ and $y$, or is $0$ (and hence has solution set all of $\mathbb{A}^2_\Q$). All but (at most) finitely many rational points on the intersection of the two corresponding varieties correspond to rational period $N$ vectors of type $\mathbf{t}$ for $f$. When $\Phi_{f,\mathbf{t},i} \neq 0$, the genus of the curve over $\Q$ defined by $\Phi_{f,\mathbf{t},i}(x,y) = 0$ obtained by any specialization of $c$ and $d$ is generically equal to the genus of the curve $Y_{\mathbf{t},i}$ over the function field $K$. 

As promised, we comment more explicitly on the above factorizations in the univariate case. First, let us consider the special case in which $N' = 0$. For $d$ a positive integer, let $\mathbf{t}_d = (L, \ldots, L)$ be the type of period $d$ consisting of all left replacements. The following factorization follows immediately from the well-known factorization in the single-variable case (see, e.g., \cite[\S 2]{FPS}):
\begin{equation}\label{fact_univariate} \Phi_{\mathbf{t}_N,L} = \prod_{d \mid N} \Phi_{\mathbf{t}_d,L},
\end{equation}
where 
\[ \Phi_{\mathbf{t}_d,L}(x,y) = \prod_{m \mid d} \left(\pi_L(f_{\mathbf{t}_m}(x,y))-x\right)^{\mu(d/m)} \in \Z[c,d,x,y] \]
is known to be an integral polynomial which is irreducible over $\Q(c,d)$. Here, $\mu$ denotes the M{\"o}bius function. These factorizations come from the fact that any point $x$ of period $d$ of a single-variable polynomial with $d \mid N$ satisfies 
\[ f^{(N)}(x) = \left(f^{(d)}\right)^{(N/d)}(x) = x. \]
Suppose more generally that $\mathbf{t}$ is univariate with 
\[ \mathbf{t} = (\underbrace{L,\ldots, L}_{N'},\underbrace{R, \ldots, R}_{N-N'}). \]
As the final $N-N'$ replacements are all on the right, we have 
 \[ \Phi_{\mathbf{t},L} = \Phi_{\mathbf{t}_{N'},L} \; , \] 
whose factorization is given by \cref{fact_univariate}. Let $\mathbf{s}_d$ denote the following type of period $N'+d$:
\[ \mathbf{s}_d := (\underbrace{L,\ldots, L}_{N'},\underbrace{R, \ldots, R}_{d}) \]
By the same reasoning which provides \cref{fact_univariate} and the fact that only the final $N-N'$ replacements are on the right, we have the factorization 
\[ \Phi_{\mathbf{t},R} = \prod_{d \mid (N-N')} \Phi_{\mathbf{s}_d,R} \; , \]
with
\[ \Phi_{\mathbf{s}_d,R}(x,y) = \prod_{m \mid d} \left(\pi_R(f_{\mathbf{s}_m}(x,y))-y\right)^{\mu(d/m)} \in \Z[c,d,x,y] \] 
integral and irreducible over $\Q(c,d)$. 

For the non-univariate types $\mathbf{t}$ which we consider in this paper, we find that the polynomials $\Phi_{\mathbf{t},i}$ for $i \in \{L,R\}$ are irreducible over $\Q(c,d)$ (equivalently, $\Psi_{\mathbf{t},i}$ is constant for $i \in \{L,R\}$). It would be of interest to know whether this is true generally of non-univariate types, but we do not investigate that question in this work. 

We next show that univariate types are more or less handled by the single-variable case of dynamics of quadratic polynomials, as claimed in the introduction. 

\begin{proposition}\label{univariate_prop}
    Let
    \[ \mathbf{t} = (\underbrace{L,\ldots, L}_{N'},\underbrace{R, \ldots, R}_{N-N'}) \]
    be a univariate type of period $N$ with $1 \leq N' \leq N$, and let $f(x,y) = Cx^2+Dy^2$ be a binary quadratic form over $\Q$. 
    \begin{enumerate}
        \item If $N' < N$, then there are no periodic vectors of type $\mathbf{t}$ for $f$. 
        \item Suppose that $N' = N$, such that 
        \[ \mathbf{t} = (L, \ldots, L). \]
        \begin{enumerate}
            \item If $N \in \{4,5\}$, then there are no periodic vectors of type $\mathbf{t}$ for $f$. 
            \item If $N \geq 6$, then \cref{Poonen_conj} implies that there are no periodic vectors of type $\mathbf{t}$ for $f$. 
        \end{enumerate}
    \end{enumerate}
\end{proposition}

\begin{proof}
    \begin{enumerate}
        \item The first $N'$ replacements for type $\mathbf{t}$ vectors are on the left, and all remaining $N-N'$ replacements are on the right. Therefore, letting 
        \[ \mathbf{s} = (\underbrace{L,\ldots, L}_{N' \textnormal{ entries}}), \]
        we note that the left modular polynomial for $\mathbf{t}$ is the same as that for $\mathbf{s}$:
        \[ \Phi_{\mathbf{t},L} = \Phi_{\mathbf{s},L}. \]
        As $\mathbf{s}$ consists of all left replacements, we have $\Phi_{\mathbf{s},R} = 0$. Therefore, any vector satisfying $\Phi_{\mathbf{t},i}$ for $i = L$ and for $i=R$ must have period at most $N' < N$.
        \item In this case, we have that $\Phi_{t,R} = 0$. Considering the polynomial $\Phi_{t,L}$ as a polynomial in the single variable $x$ over the function field $\mathbb{Q}(c,d,y)$, each specialization of $\Phi_{t,L}$ to a triple $(c,d,y) = (C,D,Y) \in \Q^3$ is a univariate modular polynomial whose solutions are generically the period $N$ points of the single variable quadratic polynomial
        \[ f(x) = Cx^2 + DY^2. \]
        The claim then follows from the referenced conjecture that single-variable quadratic polynomials have no periodic points of period at least $4$, which has been proven unconditionally for periods $4$ \cite[Thm. 4]{Mor2} and $5$ \cite[Thm. 1]{FPS}. 
    \end{enumerate}
\end{proof}

Based on \cref{univariate_prop}, the only univariate types that can admit rational periodic vectors are those equivalent to one of $(L), (L,L),$ or $(L,L,L)$. The periodic vectors in these cases will come from rational points on families of univariate dynamical modular curves, for which one can reference prior results in that setting, but we will briefly tackle these cases in \cref{L_section}, \cref{LL_section}, and \cref{LLL_section} for completeness and clarity of this point. Of greater interest to us is consideration of non-univariate types, given the novel attribute of not arising from the single-variable case.


In \cref{types_table}, for a representative of each period type class of period up to $5$ we record the degrees of the corresponding polynomials $\Phi_{f,\mathbf{t},i}$ for $i \in \{L,R\}$. We also record the genus of the affine curve $X_{\mathbf{t},i}$ cut out by $\Phi_{f,\mathbf{t},i}$ over the function field $\mathbb{Q}(c,d)$ for $i \in \{L,R\}$. We refer to the file \texttt{Yt{\_}Nlt5.m} for these computations. 

\begin{center}
\rowcolors{2}{white}{gray!20}
\begin{longtable}{| c | c | c | c | c | c |}\caption{Degrees and (when $\Phi_{\mathbf{t},i} \ne 0$) genera of $Y_{\mathbf{t},L}$ and $Y_{\mathbf{t},R}$ over $K=\mathbb{Q}(c,d)$ for types of period $N \leq 5$.}\label{types_table} \\ \hline
\rowcolor{gray!50} Type $\mathbf{t}$ & Univariate type? & $\textnormal{deg}(\Phi_{\mathbf{t},L})$ & $\textnormal{deg}(\Phi_{\mathbf{t},R})$ & $\textnormal{genus}(Y_{\mathbf{t},L})$ & $\textnormal{genus}(Y_{\mathbf{t},R})$ \\ \hline 
$(L)$ & Yes &$2$ & $0$ & $0$ & -- \\
$(L,L)$ & Yes & $2$ & $0$ & $0$ & -- \\
$(L,R)$ & Yes & $2$ & $4$ & 0 & $1$ \\
$(L,L,L)$ & Yes & $6$ & $0$ & $2$ & --\\
$(L,L,R)$ & Yes & $2$ & $8$ & $0$ & $5$ \\
$(L,L,L,L)$ & Yes & $12$& $0$ & $9$ & -- \\
$(L,L,L,R)$ & Yes & $6$& $16$ & $2$ & $17$ \\
$(L,L,R,R)$ & Yes & $2$ & $8$ & $0$ & $5$ \\
$(L,R,L,R)$ & \textcolor{blue}{No} & $8$ & $16$ & $5$ & $17$\\
$(L,L,L,L,L)$ & Yes & $30$ & $0$ & $49$ & -- \\
$(L,L,L,L,R)$ & Yes & $12$ & $32$ & $9$ & $42$ \\
$(L,L,L,R,R)$ & Yes & $6$ & $16$ & $2$ & $17$ \\
$(L,L,R,L,R)$ & \textcolor{blue}{No} & $16$ & $32$ & $17$ & $49$ \\ \hline
\end{longtable}
\end{center}

In the remainder of this section, we further study the varieties $Y_{\mathbf{t}}$ with an aim to characterize, for diagonal binary quadratic forms $f$ over $\Q$, the periodic vectors of types $\mathbf{t}$ that appear in \cref{types_table} and that are not already covered by \cref{univariate_prop}. In particular, we handle the possible univariate types in the next three subsections, and we consider the non-univariate types $(L,R,L,R)$ and $(L,L,R,L,R)$ in \cref{LRLR_section} and \cref{LLRLR_section}, respectively. 

\subsection{Period 1: self-loops}\label{L_section}

Let $f(x,y) = Cx^2+Dy^2$ be a fixed binary quadratic form over $\Q$. The next proposition states that there are infinitely many vectors $\overline{v} \in \Q^2$ such that the corresponding vertex admits a self-loop in $\mathcal{G}(f,\overline{v})$. These are in correspondence with periodic vectors of type $(L)$ or $(R)$ for $f$. By \cref{tau_bijection_lemma}, it suffices to consider type $\mathbf{t} = (L)$.  

\begin{proposition}\label{self-loops-prop}
 With $f$ as above, there exist infinitely many rational periodic vectors of type $(L)$ for $f$. 
\end{proposition}
\begin{proof}
A rational periodic vector of type $(L)$ for $f$ is a vector $\overline{v} = (a,b) \in \Q^2$ satisfying 
\[ f(\overline{v}) = a. \]
That is, they are rational points on the genus $0$ dynamical modular curve
\[ Y_{f,(L)} = Y_{f,(L),L} : Cx^2 - x + Dy^2 = 0 . \]
This curve is non-singular given that $C \neq 0$, so $Y_{f,(L)}$ is either pointless over $\Q$ or is isomorphic to $\mathbb{P}^1_{\mathbb{Q}}$. We note the trivial rational solution $(0,0)$ exists irrespective of $C$ and $D$. giving the claim. 
\end{proof}

Regarding \emph{integral} type $(L)$ vectors for $f$, we have the following result when $f$ is an integral form. 

\begin{proposition}
    \label{self_loops_integral}
Let $f = Cx^2+Dy^2$ be as above with $C,D \in \Z$. 
\begin{itemize}
    \item If $f$ is definite, then the only integral periodic vector of type $(L)$ for $f$ is $(0,0)$ unless $C = \pm 1$, in which case we also have $(\pm 1,0)$.
    \item If $f$ is indefinite, then the integral periodic $(L)$ vectors for $f$ are in one-to-one correspondence with integer solutions to the Pell equation 
    \[ X^2 - 4|CD|Y^2 = 1\]
\end{itemize}
\end{proposition}

\begin{proof}
Similar to as in \cref{self-loops-prop}, we are now looking for $\overline{v} = (x,y) \in \Z^2$ satisfying 
    \[ Cx^2 - x + Dy^2 = 0. \]
We then have 
\begin{equation}\label{Pell_solution} 
x = \dfrac{ 1 \pm \sqrt{1 - 4CDy^2}}{2C}
\end{equation}
for some choice of the sign in the numerator. Because $x$ is an integer, the quantity $1-4CDy^2$ must be a perfect square. Let $n$ be the integer satisfying 
\[ n^2 + 4CDy^2 = 1 .\]
This provides an integral solution $(X,Y) = (n,y)$ to the equation 
\[ X^2 + 4CDY^2 = 1. \]
If $C$ and $D$ have the same sign, then it follows that $X$ or $Y$ is zero and hence the only solutions are the ones of the theorem statement. 

Suppose then that we are in the indefinite case, such that $CD = -|CD|$. Note that given a pair $(n,y)$ we obtain a unique corresponding $x \in \mathbb{Z}$ such that $\overline{v}$ is periodic of type $(L)$ for $f$ via \cref{Pell_solution}: given that $1-4CDy^2$ is a square, and noting that
\[ 1-4CDy^2 \equiv  1 \pmod{2C}, \]
we see that $1 \pm \sqrt{1-4CDb^2}$ is a multiple of $2C$ for exactly one choice of the sign. Our integral period $(L)$ vectors $\overline{v}$ for $f$ are indeed then in correspondence with integral solutions to the Pell equation 
\[ n^2 - 4|CD|y^2 = 1. \]
\end{proof}

\begin{corollary}
    Let $f = Cx^2+Dy^2$ be an indefinite, binary quadratic form with $C,D \in \Z$. There are infinitely integral periodic vectors of type $(L)$ for $f$ if and only if $CD$ is not a perfect square. If $CD$ \emph{is} a perfect square, then the only possible integral periodic vector of type $(L)$ other than $(0,0)$ is $(\pm 1, 0)$ which occurs when $C = \pm 1$. 
\end{corollary}
\begin{proof}
    If $E \in \mathbb{Z}^+$ is a perfect square, then the Pell equation $X^2 - EY^2 = 1$ has only the trivial solutions $(X,Y) = (\pm 1, 0)$. If $E$ is not a perfect square, then this equation has infinitely many solutions, corresponding to norm $1$ units in $\mathbb{Z}[\sqrt{E}]$. (See, e.g., \cite[\S 20]{Dudley}.) The claim then follows quickly from the proof of \cref{self_loops_integral}. 
\end{proof}

\subsection{Period 2}\label{LL_section}

By part (1) of \cref{univariate_prop}, we know that there does not exists a quadratic form $f(x,y) = Cx^2+Dy^2$ over $\Q$ with a rational periodic vector of type $(L,R)$. The only remaining period $2$
type class to consider is then that of the univariate type $\mathbf{t} = (L,L)$.

In this case, we have 
\begin{equation}\label{LL_equation}
\Phi_{(L,L),L} = \dfrac{\pi_R(f_{(L,L)}-x)}{\Phi_{(L),L}} = c^2x^2 + cx + cdy^2+1. 
\end{equation}

The curve $Y_{(L,L)} = Y_{(L,L),L}$ over $K = \mathbb{Q}(c,d)$ is non-singular of genus $0$, as is its specialization $Y_{f,(L,L)}$ to any $f = Cx^2+Dy^2$ with $C,D \in \Q^\times$. 

\begin{proposition}\label{LL_prop}
Let $f(x,y) = Cx^2+Dy^2$ be a fixed binary quadratic form over $\Q$. Then $f$ has a rational periodic vector of type $(L,L)$ (and hence has infinitely many) if and only if there exist rational numbers $m,n$ such that 
\[ CD = -\frac{n^2+3}{4m^2}. \]
\end{proposition}
\begin{proof}
From \cref{LL_equation}, a point $(x,y)$ on the genus $0$ curve $Y_{f,(L),L}$ is given by 
\[ x = \dfrac{-1 \pm \sqrt{-(4CDy^2+3)}}{2C} .\]
So, $x \in \Q$ if and only if $-(4CDy^2+3)$ is a rational square. The vector $(x,y) \in \Q^2$ is then rational and of type $(L,L)$ if and only if we have $n \in \Q$ with
\[ CD = -\frac{n^2+3}{4y^2}. \]
Having such a point is then equivalent to having infinitely many, as $Y_{f,(L),L}$ is a non-singular conic over $\Q$. 

The above can also be quickly gleaned from the result \cite[Thm. 1]{WR} in the single-variable case (see also the restatement as \cite[Thm. 1.1]{FPS}), which uses the same line of argument. A periodic vector $(x,Y)$ of type $(L,L)$ on $f$ corresponds to a period $2$ point $x$ on the single variable polynomial 
\[ g(x) := Cx^2+DY^2, \]
treating $Y$ as fixed, which is equivalent up to conjugation by a linear polynomial to $h(x) = x^2+CDY^2$. Applying the result of Walde--Russo then completes the proof. 
\end{proof}

\subsection{Period 3}\label{LLL_section}

By part (1) of \cref{univariate_prop}, we know that there does not exist a quadratic form $f(x,y) = Cx^2+Dy^2$ over $\Q$ with a rational periodic vector of type $(L,L,R)$. The only remaining period $3$ type class to consider is then that of the univariate type $\mathbf{t} = (L,L,L)$. 

\begin{proposition}\label{LLL_prop}
    Let $f(x,y) = Cx^2+Dy^2$ be a fixed binary quadratic form over $\Q$. Then $f$ has a rational periodic vector of type $(L,L,L)$ if and only if there exist rational numbers $\tau,n$ with $\tau \not \in \{-1,0\}$ such that 
\[ CD = -\frac{\tau^6+2\tau^5+4\tau^4+8\tau^3 + 9\tau^2+4\tau+1}{4 n^2 \tau^2(\tau+1)^2}. \]
In this case, for each such pair $(\tau,n)$ there are three periodic vectors $(x,y)$ with $y = n$ of type $(L,L,L)$ which are cyclically permuted by $f_L$. 
\end{proposition}
\begin{proof}
    Similar to as in the proof of \cref{LL_prop}, we note that a $(x,Y)$ is a periodic vector of type $(L,L,L)$ for $f$ if and only if $x$ is a period $3$ point of the single-variable quadratic polynomial 
    \[ h(x) = x^2+CDY^2. \]
    By \cite[Thm. 3]{WR} (see also \cite[Thm. 1.3]{FPS}), the polynomial $h(x)$ has a period $3$ orbit consisting of rational numbers if and only if there is a rational number $\tau$ as in the statement of this proposition with 
    \[ CDY^2 = -\frac{\tau^6+2\tau^5+4\tau^4+8\tau^3 + 9\tau^2+4\tau+1}{4\tau^2(\tau+1)^2}. \]
    Our claim then follows from this correspondence and the referenced results. 
    
\end{proof}

\subsection{Period 4}\label{LRLR_section}

By parts (1) and (2)(a) of \cref{univariate_prop}, we know that there does not exist a quadratic form $f(x,y) = Cx^2+Dy^2$ over $\Q$ with a rational periodic vector of type $(L,L,L,L),$ $(L,L,L,R),$ or $(L,L,R,R)$. The only remaining period $4$ type class to consider is then that of the non-univariate type $\mathbf{t} = (L,R,L,R)$. 

In this case, we find that the polynomials 
\begin{align*} 
\Phi_{(L,R,L,R),L}(x,y) &= \pi_L(f_{(L,R,L,R)}(x,y))-x \quad \textnormal{ and } \\
\Phi_{(L,R,L,R),R}(x,y) &= \pi_R(f_{(L,R,L,R)}(x,y))-y
\end{align*}
are irreducible over $\mathbb{Q}(c,d)$ and hence are of degrees $8$ and $16$, respectively. The variety 
\[ Y_{(L,R,L,R)} = Y_{(L,R,L,R),L} \cap Y_{(L,R,L,R),R} \]
is then zero dimensional, and we may use Magma to compute its irreducible components:
\[ Y_{(L,R,L,R)}  = P_0 \cup P_{LR} \cup V. \]
Here, $P_0 = (0,0)$ and $P_{LR} = (\frac{1}{c+d},\frac{1}{c+d})$ correspond to the trivial vectors of period $1$, which are both of type $L$ and of type $R$, and $V$ is the zero-dimensional variety over $\Q(c,d)$ defined by the two equations
\begin{align}\label{LRLR_poly}
    0 = R(y) &:= 16c^2d^3(d-c)^2 y^4 + 8cd^2(d+c)(d-c)^2y^3 \\
    & \quad + (d^4+6cd^3+12c^2d^2-2c^3d-c^4)(d-c)y^2 \nonumber \\ 
    & \quad + (d^3+7cd^2-3c^2d-c^3)(d-c)y + (d^3-2cd^2+5c^2d+c^3). \nonumber
\end{align}
and 
\[ x = G(y), \]
where $G(y) \in \Q(c,d)[y]$ is an explicit cubic polynomial in $y$ over $\Q(c,d)$. These computations are handled in the file \texttt{LRLR.m}. We then see that a fixed $f(x,y) = Cx^2+Dy^2$ has a rational periodic vector of type $(L,R,L,R)$ if and only if the specialization $R_{C,D}(y)$ of $R(y)$ to $(c,d) = (C,D)$ has a rational root.

\begin{theorem}\label{LRLR_thm}
    A binary quadratic form $f(x,y) = Cx^2+Dy^2$ over $\Q$ has at most $4$ periodic vectors of type $(L,R,L,R)$ over $\C$. Such vectors exist if and only if $C \ne D$, and such vectors are in one-to-one correspondence with the roots of $R_{C,D}(y)$.

    In particular, the coordinates of a periodic vector of type $(L,R,L,R)$ for $f$ generate a number field of degree at most $4$, and $f$ has a rational periodic vector of this type over a number field $K$ if and only if $R_{C,D}(y)$ has a rational root over $K$. 
\end{theorem}
\begin{proof}
    This follows immediately from the discussion above. Specifically, given a root $y$ of $R_{C,D}(y)$ we obtain the periodic vector $(G_{C,D}(y),y)$ for $f$, where $G_{C,D}$ denotes the specialization of $G(y)$ to $f$. Note that we are already assuming that $C,D \ne 0$ throughout this work, so the condition $C \ne D$ is indeed the only one needed to guarantee periodic vectors of this type. Moreover, this condition truly guarantees vectors of \emph{exactly} this type, and not of a lower type, as we must have $C = D$ for replacements on the left and right to match. When $C=D$, note from \cref{LRLR_poly} that we indeed get no solutions. 
\end{proof}

From \cref{LRLR_thm}, it is natural to ask for which $(C,D) \in \Q^2$ with $(C,D) \ne (0,0)$ and $C \ne D$ one gets each possible factorization type of $(1,1,1,1), (1,1,2), (1,3), (2,2),$ and $(4)$ for $R_{C,D}(y)$. An argument using Hilbert's Irreducibility theorem provides that the behavior of irreducibility of $R(y)$ over $\Q(c,d)$ is exhibited by most (formally: outside of a thin set in $\Q^2$ of) specializations, and one may ask if the types other than $(4)$ occur at all. 

The answer is yes for the factorization type $(2,2)$. The following two examples show the generic behavior of factorization type $(4)$ and the less common factorization type $(2,2)$, respectively. 

\begin{example}
    Let $f(x,y) = \frac{x^2-y^2}{4}$. In this case $R_{1/4,-1/4}$ is irreducible over $\Q$ and we have
    \[ Y_{f,(L,R,L,R)} = P_0 \cup P_{LR} \cup V, \]
where $V$ is the degree $4$ point over $\Q$ given by 
\begin{align*} V: 14x &= - 3y^3 + 4y^2 - 8y - 28, \\
    0&= y^4 + 4y^2 + 16y + 28.
\end{align*}
Letting $\alpha = \sqrt{\sqrt{2}-2}$, the periodic vectors of type $(L,R,L,R)$ for $f$ consist of four vectors which are all rational over the quartic number field $\Q(\alpha)$. One cycle containing two of these vectors, along with two vectors of type $(R,L,R,L)$ is given below:
\[\begin{tikzcd}
	{\left( -\sqrt{2}-2\alpha-\sqrt{2}\alpha, \sqrt{2}\cdot (-1+\alpha)\right)} & {\left( -\sqrt{2}-2\alpha-\sqrt{2}\alpha, \sqrt{2}\cdot (-1-\alpha)\right)} \\
	{\left(-\sqrt{2}+2\alpha+\sqrt{2}\alpha, \sqrt{2}\cdot (-1+\alpha)\right)} & {\left( -\sqrt{2}+2\alpha+\sqrt{2}\alpha, \sqrt{2}\cdot (-1-\alpha) \right)} \\
	{}
	\arrow["{f_L}", from=1-1, to=2-1]
	\arrow["{f_R}", from=1-2, to=1-1]
	\arrow["{f_R}", from=2-1, to=2-2]
	\arrow["{f_L}", from=2-2, to=1-2]
\end{tikzcd}\]
\end{example}

\begin{example}
    Let $f(x,y) = x^2+6y^2$. In this case, $R_{1,6}$ has factorization type $(2,2)$ over $\Q$ and we have
    \[ Y_{f,(L,R,L,R)} = P_0 \cup P_{LR} \cup V_{2,1} \cup V_{2,2}, \]
where each of $V_{2,i}$ for $i = \{1,2\}$ is a degree $2$ point over $\Q$. Specifically, they are given by the following equations:
\begin{align*}
V_{2,1} : 8x &= 16y + 1 \\
0 &= 128y^2 + 32y + 5
\end{align*}
and 
\begin{align*}
V_{2,2} : 3x &= 9y +1 \\
0 &= 135y^2 + 45y + 7.
\end{align*}
We then have the following two cycles, each consisting of two periodic vectors of type $(L,R,L,R)$ and two of type $(R,L,R,L)$, for $f$. These correspond to the components $V_{2,1}$ and $V_{2,2}$, respectively, with the first defined over the quadratic extension $\Q(\sqrt{-6})$ and the second defined over $\Q(\sqrt{-195})$. 
\[\begin{tikzcd}
	{\left(\frac{-1+\sqrt{-6}}{8},\frac{-2+\sqrt{-6}}{16} \right)} & {\left(\frac{-1+\sqrt{-6}}{8},\frac{-2-\sqrt{-6}}{16} \right)} \\
	{\left(\frac{-1-\sqrt{-6}}{8},\frac{-2+\sqrt{-6}}{16} \right)} & {\left(\frac{-1-\sqrt{-6}}{8},\frac{-2-\sqrt{-6}}{16} \right)}
	\arrow["{f_L}", from=1-1, to=2-1]
	\arrow["{f_R}", from=1-2, to=1-1]
	\arrow["{f_R}", from=2-1, to=2-2]
	\arrow["{f_L}", from=2-2, to=1-2]
\end{tikzcd}\]

\[\begin{tikzcd}
	{\left(\frac{-5+\sqrt{-195}}{30},\frac{-15+\sqrt{-195}}{90} \right)} & {\left(\frac{-5+\sqrt{-195}}{30},\frac{-15-\sqrt{-195}}{90} \right)} \\
	{\left(\frac{-5-\sqrt{-195}}{30},\frac{-15+\sqrt{-195}}{90} \right)} & {\left(\frac{-5-\sqrt{-195}}{30},\frac{-15-\sqrt{-195}}{90} \right)}
	\arrow["{f_L}", from=1-1, to=2-1]
	\arrow["{f_R}", from=1-2, to=1-1]
	\arrow["{f_R}", from=2-1, to=2-2]
	\arrow["{f_L}", from=2-2, to=1-2]
\end{tikzcd}\]

\end{example}

The other two factorization types \emph{cannot} occur. That is, there are no binary forms of the type studied admitting a periodic vector of type $(L,R,L,R)$ over $\Q$, as the following corollary of \cref{LRLR_thm} implies.

\begin{corollary}\label{LRLR_thm_rationals}
There does not exist a binary quadratic form $f(x,y) = Cx^2+Dy^2$ over $\Q$ with a periodic vector of type $(L,R,L,R)$ over $\Q$.
\end{corollary}
\begin{proof}
    For $y = 0$, \cref{LRLR_poly} becomes 
\[ 0 = d^3 - 2cd^2 + 5c^2d+c^3 ,\]
which as an affine equation in variables $c$ and $d$ over $\Q$ defines a singular conic whose only rational point is given by $c=d=0$. Thus, we are justified in setting $z = 1/y$ and scaling by $z^4$ to rewrite \cref{LRLR_poly} as 
\begin{align}\label{LRLR_poly_z}
    0 = z^4 \cdot R(1/z) &= 16c^2d^3(d-c)^2 + 8cd^2(d+c)(d-c)^2 z \\
    & \quad + (d^4+6cd^3+12c^2d^2-2c^3d-c^4)(d-c)z^2 \nonumber \\ 
    & \quad + (d^3+7cd^2-3c^2d-c^3)(d-c)z^3 + (d^3-2cd^2+5c^2d+c^3)z^4. \nonumber
\end{align}
Taken as an equation in variables $c, d,$ and $z$ over $\Q$, \cref{LRLR_poly_z} is a homogenous polynomial of degree $7$ defining a singular projective curve $X$ in $\mathbb{P}^2_\Q$ of geometric genus $2$ with exactly $4$ singular points:
\[ (c:d:z) \in \{(0 : 0 : 1), (0 : 1 : 0), (1 : 1 : 0), (1 : 0 : 0)\}. \]
One can compute in Magma, using the \texttt{IsHyperelliptic} command, an explicit birational map $B : X \to Y$ where $Y$ is the hyperelliptic curve over $\Q$ given by the affine equation in variables $v$ and $w$
\[ v^2 + (w^3 + w^2 + w + 1) v = -w^6 - w^3 - w - 1. \]
This map $B$ is defined away from the singular points of $X$. One can compute in Magma, or find $Y$ as the genus $2$ curve with label $6400.\text{g}.64000.1$ in the LMFDB \cite{LMFDB} and see, that 
\[ Y(\Q) = \{(-1 : 0 : 1), (1 : -2 : 1)\}. \]
We compute in the file \texttt{LRLR.m} that there are no rational points in the fibers above these two points under $B$. Therefore, $X(\Q)$ consists only of the $4$ singular points listed above, none of which provide a binary quadratic form with a periodic vector of type $(L,R,L,R)$. 
\end{proof}

\subsection{Period 5}\label{LLRLR_section}
By parts (1) and (2)(a) of \cref{univariate_prop}, we know that there does not exist a quadratic form $f(x,y) = cx^2+dy^2$ over $\Q$ with a rational periodic vector of type $(L,L,L,L,L),$ $(L,L,L,L,R),$ or $(L,L,L,R,R)$. The only remaining period $5$ type class to consider is then that of the non-univariate type $\mathbf{t} = (L,L,R,L,R)$. 

Analysis for this type proceeds similarly to that for the non-univariate type $(L,R,L,R)$. Each of the two polynomials 
\begin{align*} 
\Phi_{(L,L,R,L,R),L}(x,y) &= \pi_L(f_{(L,L,R,L,R)}(x,y))-x \quad \textnormal{ and } \\
\Phi_{(L,L,R,L,R),R}(x,y) &= \pi_R(f_{(L,L,R,L,R)}(x,y))-y
\end{align*}
is irreducible over $\mathbb{Q}(c,d)$, and so the variety 
\[ Y_{(L,L,R,L,R)} = Y_{(L,L,R,L,R),L} \cap Y_{(L,L,R,L,R),R} \]
is zero dimensional. Using Magma, we compute its irreducible components to be
\[ Y_{(L,L,R,L,R)}  = P_0 \cup P_{LR} \cup W, \]
where $P_0 = (0,0)$ and $P_{LR} = (\frac{1}{c+d},\frac{1}{c+d})$ correspond to the trivial vectors of period $1$ and $W$ is the zero-dimensional variety over $\Q(c,d)$ defined by the two equations
\begin{align}\label{LLRLR_poly}
0 = S(y) &:= 256c^7d^6(c+d)^2(c^2 - 3cd + 4d^2)y^{10} + 256c^7d^6(c+d)a_9y^9 \\
& \quad + 32c^5d^3(c+d)a_8 y^8 + 16c^5d^3a_7y^7 + \sum_{i=0}^{6} a_i y^i\nonumber
\end{align}
and
\[ x = H(y), \]
where $H(y) \in \Q(c,d)[y]$ is an explicit polynomial in $y$ of degree $9$, and the $a_i \in \Z[c,d]$ are as follows
\begin{align*}
    a_0 &:= c^7 + 7c^6d + 21c^5d^2 + 37c^4d^3 - 7c^3d^4 + 91c^2d^5 - 6cd^6 + d^7, \\
    a_1 &:= c^8 - 25c^6d^2 - 59c^5d^3 - 76c^4d^4 + 190c^3d^5 + 56c^2d^6 + 10cd^7
    + d^8, \\
    a_2 &:= c^9 + c^8d + 7c^7d^2 + 104c^6d^3 + 55c^5d^4 + 12c^4d^5 + 38c^3d^6 +
    96c^2d^7 + 9cd^8 + d^9, \\
    a_3 &:= c(c^9 + 2c^8d + 8c^7d^2 + 30c^6d^3 - 192c^5d^4 +
    210c^4d^5 - 112c^3d^6 + 40c^2d^7 - 16cd^8 + 16d^9), \\
    a_4 &:= c^2(c^9 + c^8d - 14c^7d^2 - 30c^6d^3 - 146c^5d^4 +
    208c^4d^5 + 234c^3d^6 - 128c^2d^7 - 136cd^8 +
    32d^9), \\
    a_5 &:= c^3(c^9 + 2c^8d + 9c^7d^2 + 112c^6d^3 + 56c^5d^4 +
    160c^4d^5 - 32c^3d^6 - 496c^2d^7 + 48cd^8 - 96d^9), \\
    a_6 &:= c^3(c^10 + 3c^9d + 11c^8d^2 + 41c^7d^3 - 152c^6d^4 +
    264c^5d^5 + 144c^4d^6 + 240c^3d^7 - 64c^2d^8 \\
    & \quad - 288cd^9 + 64d^10), \\
    a_7 &:= c^6 - 7c^4d^2 + 54c^3d^3 + 36c^2d^4 + 8cd^5 - 32d^6, \\
    a_8 &:= c^6 + 5c^4d^2 - 14c^3d^3 + 28c^2d^4 + 40cd^5 -
    16d^6, \\
    a_9 &:= -2cd+6d^2. \\
\end{align*}
These computations are handled in the file \texttt{LLRLR.m}. 

As one may expect, these equations are a bit more intricate than those from the previous type we considered. Nonetheless, we have our analogue of \cref{LRLR_thm} for type $(L,L,R,L,R)$. 

\begin{theorem}\label{LLRLR_thm}
    A binary quadratic form $f(x,y) = Cx^2+Dy^2$ over $\Q$ has at most $10$ periodic vectors of type $(L,L,R,L,R)$ over $\C$. Any such vector corresponds to a root of the specialization $S_{C,D}(y)$ of $S(y)$ to $(c,d) = (C,D)$, and there is a one-to-one correspondence between roots of $S_{C,D}(y)$ and periodic vectors of this type if $C \ne D$. 

    The coordinates of a periodic vector of type $(L,L,R,L,R)$ for $f$ generate a number field of degree at most $10$, with degree $10$ being the generic behavior as one ranges over such $f$, and if $f$ has a periodic vector of this type over a field $K$ then $S_{C,D}(y)$ must have a root in $K$. 
\end{theorem}
\begin{proof}
    The proof is, from the computations above, as for \cref{LRLR_thm}. Note in particular that a root $Y$ of $S_{C,D}(y)$ over a number field $F$ gives by evaluation of the specialization $H_{C,D}$ of $H$ at $Y$ a unique corresponding $X \in F$ so that $(X,Y)$ is periodic for $f$ of type $(L,L,R,L,R)$ as long as $C \ne D$ (in which case the type period could be lower). The claimed generic behavior is determined by an application of Hilbert's Irreducibility Theorem, noting that $S_{C,D}(y)$ has degree $10$ as long as $C \ne -D$ and $c^2-3cd+4d^2 \ne 0$. 
\end{proof}
    
\begin{example}
    To show that the roots of $S(y)$ in general truly do not provide periodic vectors of type $(L,L,R,L,R)$ when $C=D$, we consider an example in which $S(y)$ has factorization type $(8,2)$ over $\Q$ and a root of $S(y)$ corresponds to a periodic vector of lower period. Letting $f(x,y) = x^2+y^2$, we have
    \[ Y_{f,(L,L,R,L,R)} = P_0 \cup P_{LR} \cup W_2 \cup W_8, \]
    where $W_2$ and $W_8$ are the zero-dimensional varieties over $\Q$ given by 
    \begin{align*} W_2 : \quad &x = y, \\
    &2y^2 + y + 1 = 0
    \end{align*}
    and 
    \begin{align*} W_8 : \quad &x = w(y), \\
    &1024y^8 + 512y^7 + 640y^6 - 96y^5 - 140y^4 + 81y^2 - 47y + 145 = 0 \end{align*}
    for a degree $7$ polynomial $w(y) \in \Z[y]$. 

    From $W_2$, we identify a cycle of periodic vectors over the quadratic number field $\Q(\sqrt{-7}):$

    \[\begin{tikzcd}
	{\left( \dfrac{-1 + \sqrt{-7}}{4} , \dfrac{-1 + \sqrt{-7}}{4} \right)} & {\left( \dfrac{-3 - \sqrt{-7}}{4} , \dfrac{-1 + \sqrt{-7}}{4} \right)} \\
	{\left( \dfrac{-1 + \sqrt{-7}}{4} , \dfrac{-3 - \sqrt{-7}}{4} \right)}
	\arrow["{f_L}", shift left=2, from=1-1, to=1-2]
        \arrow["{f_R}"', from=1-2, to=1-2, loop, in=300, out=240, distance=5mm]
	\arrow["{f_R}"', from=1-1, to=2-1]
	\arrow["{f_L}", shift left, from=1-2, to=1-1]
	\arrow["{f_R}"', shift right=5, from=2-1, to=1-1]
	\arrow["{f_L}"', from=2-1, to=2-1, loop, in=300, out=240, distance=5mm]
\end{tikzcd}\]

    Note that one of these vectors is periodic both of type $(L,L)$ and of type $(R,R)$ while the other two are periodic of type $(L)$ and periodic of type $(R)$ individually. Hence, no element of the orbit is of type $(L,L,R,L,R)$. 
\end{example}

It would be interesting if one could determine whether there are any specializations of $S_{C,D}$ of the polynomial $S$ admitting a rational root, or more generally admitting a different factorization type over $\Q$. One can in theory run the same argument as in proof of \cref{LRLR_thm_rationals} to try to determine the existence of rational periodic vectors of this type; there are no such vectors with $y = 0$, and so such a vector would correspond to a non-singular point on the projective curve $X$ over $\Q$ defined by the following homogenous equation in variables $c, d,$ and $z = 1/y$ (with $S(y)$ as in \cref{LLRLR_poly}):
\[ X : z^{10}S(1/y) = 0.  \]
One then would like to determine whether this curve $X$, which has geometric genus $11$ and is not birational to a hyperelliptic curve, has rational points other than its $4$ singular points. Computing a birational model for $X$ to investigate its rational points, or otherwise determining the rational points of $X$, seems like a difficult task given the large genus of $X$ and the intricacy of its defining equation. We explicitly compute in Magma (see the file \texttt{LLRLR.m}) that $X$ has no smooth rational points of height up to $10000000$, though, and for this reason we feel content in making the following conjecture. 

\begin{Conjecture}\label{LLRLR_conjecture}
There exist no distinct, nonzero rational numbers $C, D$ such that the binary quadratic form $f(x,y) = Cx^2+Dy^2$ has a rational periodic vector of type $(L,L,R,L,R)$ (equivalently, such that the polynomial $S_{C,D}(y)$ has a rational root). 
\end{Conjecture}

\bibliographystyle{amsalpha}
\bibliography{biblio}
\end{document}